\documentclass[12pt]{amsart} 
\usepackage{amsmath,amssymb,amsthm} 
\usepackage{color,colordvi,graphicx,verbatim,MnSymbol} 
\usepackage[noend]{algorithmic} 
\usepackage{filecontents}
 
\newtheorem{theorem}{Theorem}

\newtheorem{proposition}[theorem]{Proposition} 
 
\theoremstyle{definition} 
 
\newtheorem{example}[theorem]{Example} 
\newtheorem{remark}[theorem]{Remark} 
\newtheorem{algorithm}[theorem]{Algorithm} 
 
 \newcommand{\ellx}{{\ell^{(1)}}}    % {{\ell_{1,0}}  
 \newcommand{\elly}{{\ell^{(2)}}}    % {{\ell_{0,1}}  
 \newcommand{\hx}{{h^{(1)}}}         % {{h_{1,0}}} 
 \newcommand{\hy}{{h^{(2)}}}         % {{h_{0,1}}} 
  \newcommand{\ellxPrime}{{k^{(1)}}} % { \ell'_{1,0}(x)} 
  \newcommand{\ellyPrime}{{k^{(2)}}} % { \ell'_{0,1}(y)} 

  \newcommand{\MOne}{{M_{(1)}}} % {M_1} 
  \newcommand{\MTwo}{{M_{(2)}}} % {M_2} 
  \newcommand{\Mi}{{M_{(i)}}} % {M_i} 

  \newcommand{\MOnePrime}{{M_{(1)}'}} % {M_1'} 
  \newcommand{\MTwoPrime}{{M_{(2)}'}} % {M_2'} 
  \newcommand{\MiPrime}{{M_{(i)}'}}   % {M_i'} 

\newcommand{\ratto}{\dashedrightarrow} 
 
\newcommand{\V}[1]{{\mathbb{V} \! \left(#1\right)}} % variety 
 
\newcommand{\CC}{{\mathbb{C}}} 
\newcommand{\PP}{{\mathbb{P}}}

\DeclareMathOperator{\trace}{tr} 
\DeclareMathOperator{\Proj}{Proj}

\newcommand{\defcolor}[1]{{\color{blue}#1}} 
\newcommand{\demph}[1]{\defcolor{{\sl #1}}}

\title{Trace test} % done right 
\author[A.~Leykin]{Anton Leykin} 
\address{Anton Leykin\\
         School of Mathematics\\ 
         Georgia Institute of Technology\\ 
         686 Cherry Street\\ 
         Atlanta, GA 30332-0160 USA\\ 
         USA} 
\email{leykin@math.gatech.edu} 
\urladdr{http://people.math.gatech.edu/~aleykin3} 
\author[J.~I.~Rodriguez]{Jose Israel Rodriguez}
\address{Jose Israel Rodriguez\\
         Department of Statistics\\
         University of Chicago\\
         Chicago, IL 60637\\         
         USA}
\email{JoIsRo@uchicago.edu}
\urladdr{http://home.uchicago.edu/\~joisro}
\author[F.~Sottile]{Frank Sottile} 
\address{Frank Sottile\\ 
         Department of Mathematics\\ 
         Texas A\&M University\\ 
         College Station\\ 
         Texas \ 77843\\ 
         USA} 
\email{sottile@math.tamu.edu} 
\urladdr{http://www.math.tamu.edu/~sottile} 
\thanks{Research of Leykin supported in part by NSF grant DMS-1151297} 
\thanks{Research of Rodriguez supported in part by NSF grant DMS-1402545}         
\thanks{Research of Sottile supported in part by NSF grant DMS-1501370} 
\begin{document} 
 
%%%%%%%%%%%%%%%%%%%%%%%%%%%%%%%%%%%%%%%%%%%%%%%%%%%%%%%%%%%%%%%%%%%%%%%%%%%% 
\begin{abstract} 
The trace test in numerical algebraic geometry verifies the completeness of a witness
set of an irreducible variety in affine or projective space. 
We give a brief derivation of the trace test and then consider it for subvarieties of products of projective spaces using multihomogeneous witness sets. 
We show how a dimension reduction %based on Bertini's Theorem 
leads to a practical trace test in this case involving a curve in a low-dimensional affine space. 
\end{abstract} 
%%%%%%%%%%%%%%%%%%%%%%%%%%%%%%%%%%%%%%%%%%%%%%%%%%%%%%%%%%%%%%%%%%%%%%%%%%%% 
\maketitle 
%%%%%%%%%%%%%%%%%%%%%%%%%%%%%%%%%%%%%%%%%%%%%%%%%%%%%%%%%%%%%%%%%%%%%%%%%%%% 
 
%%%%%%%%%%%%%%%%%%%%%%%%%%%%%%%%%%%%%%%%%%%%%%%%%%%%%%%%%%%%%%%%%%%%%%%%%%%% 
\section*{Introduction} 
Numerical algebraic geometry~\cite{SW05} uses numerical analysis to study algebraic varieties, which are sets defined by
polynomial equations.
It is becoming a core tool in applications of algebraic geometry outside of mathematics.   
Its fundamental concept is a witness set, which is a 
general linear section of an algebraic variety~\cite{SV}. 
This gives a representation of a variety which may be manipulated on a computer and forms the basis for many
algorithms.
The trace test is used to verify that a witness set is complete. 

We illustrate this with the folium of Descartes, defined by $x^3+y^3=3xy$.
A general line $\ell$ meets the folium in three points $W$ and the pair $(W,\ell)$ forms a \demph{witness set} for
the folium.
Tracking the points of $W$ as $\ell$ moves computes witness sets on other lines.
Figure~\ref{F:folium} shows these witness sets on four parallel lines.
%%%%%%%%%%%%%%%%%%%%%%%%%%%%%%%%%%%%%%%%%%%%%%%%%%%%%%%%%%%%%%%%%%%%%%%%%%%%%%%%%
\begin{figure}[htb]
\[
   \begin{picture}(182,92)%(-30,0)
     \put(0,0){\includegraphics{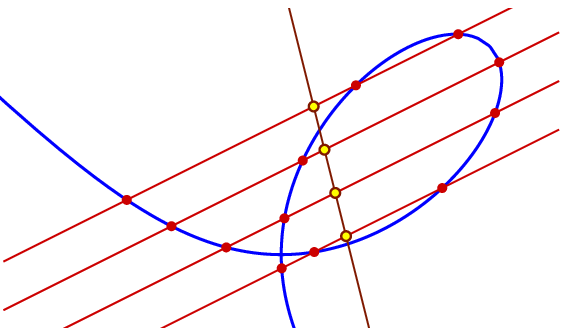}}
     \put(-30,72){$x^3+y^3=3xy$}
     \put(163,53){$\ell$}
     \put(125,7){collinear traces}
     \put(123,10){\vector(-1,0){16}}
    \end{picture}
\]
\caption{Witness sets and the trace test for the folium of Descartes}
\label{F:folium}
\end{figure}
%%%%%%%%%%%%%%%%%%%%%%%%%%%%%%%%%%%%%%%%%%%%%%%%%%%%%%%%%%%%%%%%%%%%%%%%%%%%%%%%%
It also shows the average of each witness set, which is one-third of their sum, the \demph{trace}.
The four traces are collinear.

Any subset $W'$ of $W$ may be tracked to get a corresponding subset on any other line, and we may
consider the traces of the subsets as $\ell$ moves in a pencil.
The traces are collinear if and only if $W'$ is \demph{complete} in that $W'=W$. 
This may also be seen in Figure~\ref{F:folium}.
This \demph{trace test}~\cite{SVW_trace} is used to verify the completeness of a subset of a witness set.

Methods to check linearity of a univariate function---e.g., the trace---in the context of
algorithms for numerical algebraic geometry were recently discussed in~\cite{BHL16}. 

An algebraic variety $V$ may be the union of other varieties, called its components.
Given a witness set $W=V\cap L$ for $V$ ($L$ is a linear space), \demph{numerical irreducible
  decomposition}~\cite{SVW_decomposition} partitions $W$ into subsets corresponding to the 
components of $V$.
For example, suppose that $V=E\cup F$ is the union of the ellipse $8(x+1)^2+ 3(2y+x+1)^2 = 8$ and the folium, as in
Figure~\ref{F:folium_ellipse}.
%%%%%%%%%%%%%%%%%%%%%%%%%%%%%%%%%%%%%%%%%%%%%%%%%%%%%%%%%%%%%%%%%%%%%%%%%%%%%%%%%
\begin{figure}[htb]
\[
   \begin{picture}(192,108)%(-30,0)
     \put(0,0){\includegraphics{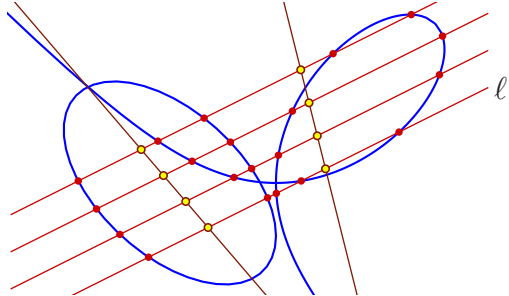}}
     \put(184,74){$\ell$}
%     \put(-30,72){$x^3+y^3=3xy$}
%     \put(163,52){$\ell$}
%     \put(125,7){collinear traces}
%     \put(123,10){\vector(-1,0){16}}
    \end{picture}
\]
\caption{Numerical irreducible decomposition for the ellipse and folium}
\label{F:folium_ellipse}
\end{figure}
%%%%%%%%%%%%%%%%%%%%%%%%%%%%%%%%%%%%%%%%%%%%%%%%%%%%%%%%%%%%%%%%%%%%%%%%%%%%%%%%%
A witness set for $V$ consists of the five points $W=V\cap\ell$.
Tracking points of $W$ as $\ell$ varies in a loop in the space of lines, a point $w\in W$ may move to a
different point which lies in the same component of $V$.
Doing this for several loops partitions $W$ into two sets, of cardinalities two and three, respectively.
Applying the trace test to each subset verifies that each is a witness set of a component of $V$.
\smallskip

A multiprojective variety is subvariety of a product of projective spaces.
Since there are different types of general linear sections in a product of projective spaces,
a witness set for a  multiprojective variety is necessarily a collection of such sections, called a witness collection. 
We see this in Figure~\ref{F:bilinear_noGreen}, where vertical and horizontal lines are the two types of hyperplanes in 
the product $\PP^1\times\PP^1$. 
%%%%%%%%%%%%%%%%%%%%%%%%%%%%%%%%%%%%%%%%%%%%%%%%%%%%%%%%%%%%%%%%%%%%%%%%%%%%%%%%%
\begin{figure}[htb]
  \begin{picture}(183,148)(-23,0)
   \put(0,0){\includegraphics[height=150pt]{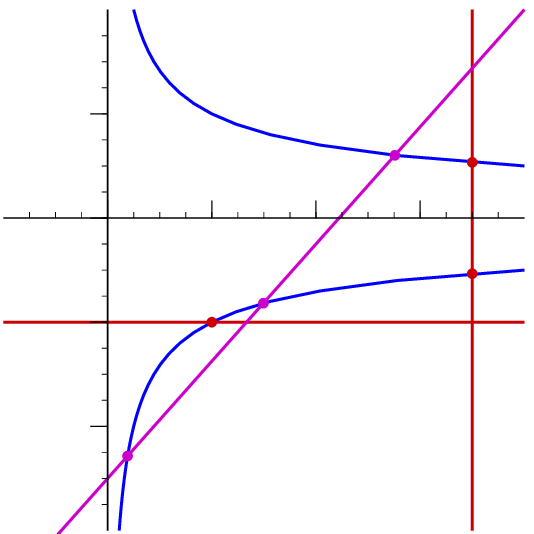}}
   \put(10,26.5){$-2$}   \put(18,116){$1$}  \put(18,145){$y$}
   \put(57,97){$1$} \put(87,97){$2$}  \put(142,94){$x$}
   \put(-11,62){$\elly$} 
   \put(55,35){$\ell$}
   \put(137,27){$\ellx$}  \put(152,71){$C$} \put(152,102){$C$}
  \end{picture}
   \qquad
  \begin{picture}(183,148)(-23,0)
   \put(0,0){\includegraphics[height=150pt]{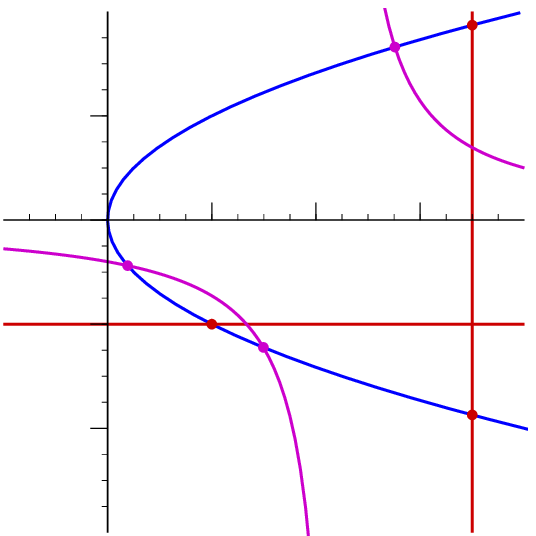}}
   \put(10,26.5){$-2$}   \put(18,116){$1$}  \put(0,145){$z=\frac{1}{y}$}
   \put(57,97){$1$} \put(87,97){$2$} \put(117,97){$3$}  \put(146,94){$x$}
   \put(-11,62){$\elly$} 
   \put(77,15){$\ell$} \put(110,116){$\ell$}
   \put(137,8){$\ellx$}  \put(78,132){$C$}
  \end{picture}
 \caption{The curve $C$ in $\PP^1\times\PP^1$ is defined by  $xy^2=1$ and $x = z^2=(1/y)^2$ in two distinct affine
   charts $\CC\times\CC$. 
    Its witness collection is $((W_1,\ell^{(1)}),(W_2,\ell^{(2)}))$, where $W_i=C\cap\ell^{(i)}$.}
 \label{F:bilinear_noGreen}
\end{figure}
%%%%%%%%%%%%%%%%%%%%%%%%%%%%%%%%%%%%%%%%%%%%%%%%%%%%%%%%%%%%%%%%%%%%%%%%%%%%%%%%%

Witness sets for multihomogeneous varieties were introduced in~\cite{HR15}. 
Figure~\ref{F:bilinear_noGreen} shows that the trace obtained by varying $\ell^{(2)}$ is nonlinear in either 
affine chart.  
One may instead apply the trace test to a witness set 
in the ambient projective space of the Segre embedding.  
By Remark~\ref{remark:Segre}, this may involve very large witness sets. 
We propose an alternative method to verify irreducible components, using a dimension reduction that sidesteps this potential
bottleneck followed by  the ordinary trace test in an affine patch on the product of projective spaces. 
In Figure~\ref{F:bilinear_noGreen} this is represented by the linear section of the plane cubic $xy^2=1$ by the line
$\ell$. 
Both $xy^2=1$ and $x=z^2=(\frac{1}{y})^2$ bihomogenize to the same cubic, but line $\ell$ in the first affine chart becomes a quadric in the second. Moreover, a general line in the second chart intersects the curve at two points. 
Taking a generic chart preserves the total degree, so we first choose a chart, and then
take a general linear section in that chart.
This will have the same number of points as the the total number of points in the witness collection.
 
In \S\ref{S:affine} we present a simple derivation of the usual trace test in affine space. 
While containing the same essential ideas as in~\cite{SVW_trace}, our derivation is shorter, 
and we believe significantly clearer. 
In \S\ref{S:product} we introduce witness collections, collections of multihomogeneous witness sets representing multiprojective varieties.
In \S\ref{S:dimension} we present 
trace test for multihomogeneous varieties that exploits a 
reduction in dimension. 
Proofs are placed in \S\ref{S:proofs} to streamline the exposition.

%%%%%%%%%%%%%%%%%%%%%%%%%%%%%%%%%%%%%%%%%%%%%%%%%%%%%%%%%%%%%%%%%%%%%%%%%%%%%%%%% 
\section{Trace in an affine space}\label{S:affine}   

We derive the trace test for curves in affine space, which verifies the completeness of a witness set.
We also show how to reduce to a curve when the variety has greater dimension.
Let $\defcolor{V}\subset\CC^n$ be an irreducible algebraic variety of dimension $m>0$. 
We restrict to $m>0$, for if $m=0$, then $V$ is a single point. 
Let $(x,y)$ be coordinates for $\CC^n$ with $x\in\CC^{n-m}$ and $y\in\CC^m$. 
Polynomials defining $V$ generate a prime ideal $I$ in the polynomial ring $\CC[x,y]$. 
We assume that $V$ is in general position with respect to these coordinates. 
In particular, the projection $\pi$ of $V$ to $\CC^m$ is a branched cover with a fiber of 
$\defcolor{d}=\deg V$ points outside the ramification locus $\defcolor{\Delta}\subset\CC^m$. 

%%%%%%%%%%%%%%%%%%%%%%%%%%%%%%%%%%%%%%%%%%%%%%%%%%%%%%%%%%%%%%%%%%%%%%%%%%%%%%%%% 
\begin{example}\label{ex:projectFolium}
 If we project the folium of Descartes to the $y$-axis, all fibers consist of three points, except those above 
 zeroes of the discriminant $-27y^3(y^3-4)$.
 These zeroes form the ramification locus
 $\Delta=\{0,2^{2/3}, (-\frac{1}{2}\pm\frac{\sqrt{-3}}{2})2^{2/3}\}$.
%%%%%%%%%%%%%%%%%%%%%%%%%%%%%%%%%%%%%%%%%%%%%%%%%%%%%%%%%%%%%%%%%%%%%%%%%%%%%%%%%
\begin{figure}[htb]
  \begin{picture}(216,112)
   \put(0,0){\includegraphics{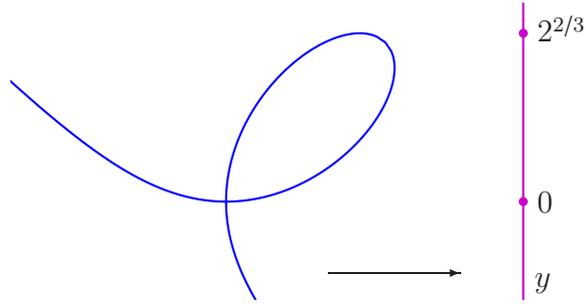}}
   \put(120,10){\vector(1,0){50}}
   \put(199,33){$0$}   \put(199,98){$2^{2/3}$}
   \put(198,5){$y$}
  \end{picture}
  \caption{Projecting the folium to the $y$-axis.}
 \label{F:Branch_Locus}
\end{figure}
%%%%%%%%%%%%%%%%%%%%%%%%%%%%%%%%%%%%%%%%%%%%%%%%%%%%%%%%%%%%%%%%%%%%%%%%%%%%%%%%%
 Figure~\ref{F:Branch_Locus} shows the real points, where the fiber consists of one or three points, with this number
 changing at the real points of~$\Delta$.  
\end{example}
%%%%%%%%%%%%%%%%%%%%%%%%%%%%%%%%%%%%%%%%%%%%%%%%%%%%%%%%%%%%%%%%%%%%%%%%%%%%%%%%%
 
Let $\defcolor{\ell}\subset\CC^m$ be a general line parameterized by $t\in\CC$, so that 
$\defcolor{L}:=\CC^{n-m}\times\ell$ is a general affine subspace of dimension $n{-}m{+}1$ with coordinates 
$(x,t)$. 
The intersection $\defcolor{C}:=V\cap L$ is an irreducible curve of degree $d$ by Bertini's 
Theorem~(see Theorem~\ref{thm:bertini}) and the projection $\pi\,\colon C\to\ell$ is a degree $d$ cover over 
$\ell\smallsetminus\Delta$.

%%%%%%%%%%%%%%%%%%%%%%%%%%%%%%%%%%%%%%%%%%%%%%%%%%%%%%%%%%%%%%%%%%%%%%%%%%%%%%%%% 
 \iffalse
\begin{example}\label{ex:circle}
{\color{blue}
Let $V$ denote the sphere defined by $x^2+y^2+z^2=1$.
The branch locus $\triangle$ of the projection $\pi$ of $V$ to the $yz$-plane is the circle $y^2+z^2=1$.  
The line $\ell\subset\CC_{y,z}^2$ induces a hyperplane $L$ in $\CC^3_{x,y,z}$.  The intersection $C:=L\cap V$ is a circle.
%If $\alpha$ is a generic projection, then $\alpha(C)$ is irreducible (but need not be a circle, e.g. $\alpha(C)$ may be an ellipse).
%More importantly, according to Prop.~\ref{prop:generic-projection-to-P2}, we are able to determine irreducibility of $C$ by determining irreducibility of $\alpha(C)$.
}
{\color{red} Include picture?}
\end{example}
\fi
%%%%%%%%%%%%%%%%%%%%%%%%%%%%%%%%%%%%%%%%%%%%%%%%%%%%%%%%%%%%%%%%%%%%%%%%%%%%%%%%% 

%%%%%%%%%%%%%%%%%%%%%%%%%%%%%%%%%%%%%%%%%%%%%%%%%%%%%%%%%%%%%%%%%%%%%%%%%%%%%%%%% 
\begin{proposition}\label{prop:generic-projection-to-P2} 
 Let $C\subset \PP^n$, $n\geq 2$, be a curve. Let $\alpha\,\colon\PP^n\ratto \PP^2$ be a generic projection. 
 Then $C$ is irreducible if and only if $\alpha(C)$ is irreducible. 
\end{proposition} 
%%%%%%%%%%%%%%%%%%%%%%%%%%%%%%%%%%%%%%%%%%%%%%%%%%%%%%%%%%%%%%%%%%%%%%%%%%%%%%%%% 
Note that one can state Proposition~\ref{prop:generic-projection-to-P2} for a generic linear map of affine spaces $\CC^n\to\CC^2$, since taking affine charts preserves (ir)reducibility.
 
Since $V$ and $L$ are in general position, Proposition~\ref{prop:generic-projection-to-P2} implies that 
the projection of $C$ to the  $(x_i,t)$-coordinate plane 
is an irreducible curve given by a single polynomial \defcolor{$f(x_i,t)$} of degree $d$ with all 
monomials up to degree $d$ having nonzero coefficients.

 %%%%%%%%%%%%%%%%%%%%%%%%%%%%%%%%%%%%%%%%%%%%%%%%%%%%%%%%%%%%%%%%%%%%%%%%%%%%%%%%% 
\iffalse
 \begin{example}\label{ex:ellipse}
{\color{blue}
%According to Prop.~\ref{prop:generic-projection-to-P2}, we are able to determine irreducibility of $C$ by determining irreducibility of $\alpha(C)$.
Continuing Ex.~\ref{ex:circle}, we let $\alpha:\CC^3\to\CC^2$ denote a general projection. Then $\alpha(C)$ is an ellipse and irreducible.
According to Prop.~\ref{prop:generic-projection-to-P2}, the curve $C$ is also irreducible. 
%Likewise, we see this with the folium.
}
{\color{red} Include picture to illustrate this idea?}
\end{example}
\fi
%%%%%%%%%%%%%%%%%%%%%%%%%%%%%%%%%%%%%%%%%%%%%%%%%%%%%%%%%%%%%%%%%%%%%%%%%%%%%%%%% 

Normalize $f$ so that the coefficient of $x_i^d$ is $1$, and extend scalars from $\CC$ to $\CC(t)$. 
Then $f\in \CC(t)[x_i]$ is a monic irreducible polynomial in $x_i$. 
The negative sum of its roots is the coefficient of $x_i^{d-1}$ in $f$, which is an affine function 
of $t$. 
Equivalently, 
 \[ 
   \trace_{K/\CC(t)}(x_i)\ =\ c_0 t+c_1\,, \quad(\mbox{for some }c_0,c_1\in \CC)\,, 
 \] 
where $K$ is a finite extension of $\CC(t)$ containing the roots of $f$. 
A function of $t$ of the form $c_0 t+c_1$ where $c_0,c_1$ are constants is an 
\demph{affine function}. 
We deduce the following. 
 
%%%%%%%%%%%%%%%%%%%%%%%%%%%%%%%%%%%%%%%%%%%%%%%%%%%%%%%%%%%%%%%%%%%%%%%%%%%%%%%%% 
\begin{proposition}\label{Prop:one} 
 The sum in $\CC^{n-m}$ of the points in a fiber of $C$ over $t\in\ell\smallsetminus\Delta$ is an affine 
 function of $t$. 
\end{proposition} 
%%%%%%%%%%%%%%%%%%%%%%%%%%%%%%%%%%%%%%%%%%%%%%%%%%%%%%%%%%%%%%%%%%%%%%%%%%%%%%%%% 
 
The converse to this holds. 
 
%%%%%%%%%%%%%%%%%%%%%%%%%%%%%%%%%%%%%%%%%%%%%%%%%%%%%%%%%%%%%%%%%%%%%%%%%%%%%%%%% 
\begin{proposition}\label{Prop:two} 
 No proper subset of the points in a fiber of $C$ over $t\in\ell\smallsetminus\Delta$ has sum that is an affine 
 function of $t$. 
\end{proposition} 
%%%%%%%%%%%%%%%%%%%%%%%%%%%%%%%%%%%%%%%%%%%%%%%%%%%%%%%%%%%%%%%%%%%%%%%%%%%%%%%%% 

\begin{example}\label{ex:foliumAlpha}
 Consider this for the folium of Descartes.
 As the folium is a plane curve, a general projection $\alpha$ of Proposition~\ref{prop:generic-projection-to-P2} is a
 general change of coordinates. 
 In the  coordinates $\xi=x+1$ and $t=2y-x$, the folium has equation
\[
   9\xi^3+(3t-39)\xi^2+(3t^2-18t+51)\xi+t^3-3t^2+15t-21\ =\ 0\,.
\]
 The trace is $-\frac{t}{3}+ \frac{13}{3}$.
 Figure~\ref{F:New_Folium} shows Figure~\ref{F:folium} under this change of coordinates.
\begin{figure}[htb]
   \includegraphics{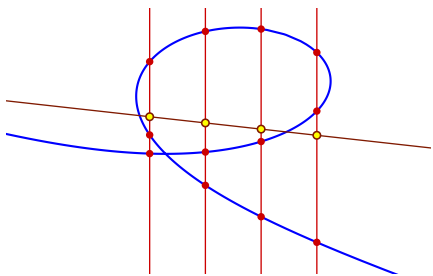}
 \caption{Folium of Descartes in new coordinates}
 \label{F:New_Folium}
\end{figure}
 The lines become vertical, and the average of the trace is the line
 $\xi=-\frac{t}{9}+ \frac{13}{9}$.
\end{example}
 
%%%%%%%%%%%%%%%%%%%%%%%%%%%%%%%%%%%%%%%%%%%%%%%%%%%%%%%%%%%%%%%%%%%%%%%%%%%%%%%%% 
\begin{remark} 
 We generalize the situation of Proposition~\ref{Prop:two}. 
 A \demph{pencil} of linear spaces is a family \defcolor{$M_t$} for $t\in\CC$ of linear spaces that depends 
 affinely on the parameter $t$. 
 Each $M_t$ is the span of a linear space $L$ and a point $t$ on a line $\ell$ that is disjoint from $L$. 
 
 Suppose that $V\subset\PP^n$ is a subvariety of dimension $m$ and that $M_t$ for $t\in\CC$ is a general pencil of 
 linear subspaces of codimension $m$ with $V\cap M_0$ transverse. 
 Let $\Delta\subset\CC$ be the finite set of points $t$ such that the intersection $V\cap M_t$ is not transverse. 
 Given any path $\gamma\,\colon[0,1]\to\CC\smallsetminus\Delta$ with $\gamma(0)=0$ and any $v\in V\cap 
 M_0$, we may analytically continue $v$ along $\gamma$ to obtain a path $v(\gamma(s))$ for $s\in[0,1]$ 
 with $v(\gamma(s))\in V\cap M_{\gamma(s)}$. 
 
 The sum of the points in a subset $W$ of $V\cap M_0$ is 
 \demph{an affine function of $t$} if for a nonconstant path $\gamma\,\colon[0,1]\to\CC\smallsetminus\Delta$ with 
 $\gamma(0)=0$, the sum of the points $w(\gamma(s))$ is an affine function of $\gamma(s)$. 
 This is independent of choice of path and of a general pencil. 
\end{remark} 
%%%%%%%%%%%%%%%%%%%%%%%%%%%%%%%%%%%%%%%%%%%%%%%%%%%%%%%%%%%%%%%%%%%%%%%%%%%%%%%%% 
 
%%%%%%%%%%%%%%%%%%%%%%%%%%%%%%%%%%%%%%%%%%%%%%%%%%%%%%%%%%%%%%%%%%%%%%%%%%%%%%%%% 
\begin{remark}\label{remark:trace-test} 
%{\color{red} Can we use the branch cover to bring in monodromy here? Anton: I would not... I eliminated the reference
%to ``monodromy'' altogether -- how we get partial witness sets is alluded to in the NID discussion.} 
 This leads to the trace test. 
 Let $V\subset\PP^n$ (or $\CC^n$) be a possibly reducible variety of dimension $m$ and $M$ a general 
 linear space of codimension $m$ so that $W=V\cap M$ is a witness set for $V$. 
 Suppose we have 
 a subset $\emptyset\neq W'\subset W$ whose points lie in a single 
 component $V'$ of $V$  so that $W'\subset V'\cap M$. 
 Such a set $W'$ is a \demph{partial witness set} for $V'$. 
 To test if $W'=V'\cap M$, let $M_t$ for $t\in\CC$  be a general pencil of codimension $m$ planes in 
 $\PP^n$ with $M=M_0$ and %we 
 test if the sum of the points of $W'$ is an affine function of $t$. 
 By Proposition~\ref{Prop:two}, $W'=V'\cap M$ if and only if it passes this \demph{trace test}. 
\end{remark} 
%%%%%%%%%%%%%%%%%%%%%%%%%%%%%%%%%%%%%%%%%%%%%%%%%%%%%%%%%%%%%%%%%%%%%%%%%%%%%%%%% 
 
%%%%%%%%%%%%%%%%%%%%%%%%%%%%%%%%%%%%%%%%%%%%%%%%%%%%%%%%%%%%%%%%%%%%%%%%%%%%%%%%% 
\begin{remark}\label{remark:trace-image} 
 Let $U$ be a variety and $\phi\,\colon U\ratto \PP^n$ be a rational map with image $V=\overline{\phi(U)}$. 
 As obtaining defining equations for $V$ may not be practical, working with a witness set $V\cap M$ may not be feasible. 
Instead one may work with the preimage $\phi^{-1}(V\cap M)$ producing a proxy for the witness set $V\cap M$.  
A \demph{partial proxy witness set} is a finite subset of $\phi^{-1}(V\cap M)$. 
It is \demph{complete} if its image is a complete witness set.
 
We can, in particular, employ the trace test for the image working with proxy witness sets
for $V\cap M_t$ in Remark~\ref{remark:trace-test}. 
\end{remark} 
 
Hauenstein and Sommese use this general observation in \cite{WitnessProj} to provide a detailed description of
how proxy witness sets can be computed and used to get witness sets of images of subvarieties under a linear
map $\PP^m \to \PP^n$.  
 
%%%%%%%%%%%%%%%%%%%%%%%%%%%%%%%%%%%%%%%%%%%%%%%%%%%%%%%%%%%%%%%%%%%%%%%%%%%%%%%%% 
\section{Witness collections for multiprojective varieties}\label{S:product} 
 
Suppose that $V\subset\CC^{n_1}\times\CC^{n_2}$ is an irreducible variety of dimension $m>0$. 
Letting $z^{(i)}$ be coordinates for $\CC^{n_i}$ for $i=1,2$, the variety $V$ is defined by polynomials 
$F(z^{(1)},z^{(2)})$ which generate a prime ideal. 
Separately homogenizing these polynomials in each set $z^{(i)}$ of variables gives bihomogeneous polynomials 
that define the closure $\overline{V}$ of $V$ in the product $\PP^{n_1}\times\PP^{n_2}$ of projective spaces. 
Let us also write $V$ for this closure. 
 
Then $V$ has a \demph{multidegree}~\cite[Ch.~19]{Harris}. 
This is a set of nonnegative integers $d_{m_1,m_2}$ where $m_1+m_2=m$ with $0\leq m_i\leq n_i$ for 
$i=1,2$ that has the following geometric meaning. 
Given general linear subspaces $M_i\subset\PP^{n_i}$ of codimension $m_i$ for $i=1,2$ with 
$m_1+m_2=m$, the number of points in the intersection $V\cap (\MOne\times \MTwo)$ is $d_{m_1,m_2}$. 
Multidegrees are log-concave in that for every $1\leq m_1\leq m{-}1$, we have 
 \begin{equation}\label{Eq:log-concave} 
   d_{m_1,m_2}^2\ \geq\ d_{m_1-1,m_2+1}\cdot d_{m_1+1,m_2-1}\,. 
 \end{equation} 
These inequalities of Khovanskii and Tessier are explained in~\cite[Ex.~1.6.4]{Laz}. 
 
Following~\cite{HR15}, a \demph{multihomogeneous witness set} of dimension $(m_1,m_2)$ with $m_1+m_2=\dim V$ 
for an irreducible variety $V$ is a set $\defcolor{W_{m_1,m_2}}:=V\cap(\MOne\times \MTwo)$, where for $i=1,2$, 
$\Mi\subset\PP^{n_i}$ is a general linear subspace of codimension $m_i$. 
More formally, the witness set is a triple consisting of the points $W_{m_1,m_2}$, 
equations for a variety that has $V$ as a component, and equations for  $\MOne$ and for $\MTwo$. 
A \demph{witness collection} is the list of witness sets $W_{m_1,m_2}$ for all $m_1+m_2=m$. 
 
%%%%%%%%%%%%%%%%%%%%%%%%%%%%%%%%%%%%%%%%%%%%%%%%%%%%%%%%%%%%%%%%%%%%%%%%%%%%%%%%% 
\begin{remark}\label{R:WitnessProj} 
 If $m_2=\dim \pi_2(V)$, then the multihomogeneous witness set $V\cap(\MOne\times \MTwo)$ is a (proxy) witness 
 set for the image $\pi_2(V)$ of $V$ in the sense of~\cite{WitnessProj} and
 Remark~\ref{remark:trace-image}. 
\end{remark} 
%%%%%%%%%%%%%%%%%%%%%%%%%%%%%%%%%%%%%%%%%%%%%%%%%%%%%%%%%%%%%%%%%%%%%%%%%%%%%%%%% 
 
Suppose that $V$ is reducible and $W_{m_1,m_2}=V\cap(\MOne\times \MTwo)$  is a multihomogeneous witness set for 
$V$. 
This is a disjoint union of multihomogeneous witness sets for the irreducible components of $V$ that have non zero 
$(m_1,m_2)$-multidegree. 
We similarly have a witness collection for $V$. 
We consider the problem of decomposing a witness collection into witness collections for the components 
of $V$. 
For every irreducible component $V'$ of $V$ it is possible to obtain a \demph{partial witness collection} 
$W'_{m_1,m_2}$ for $m_1+m_2=m$ and then---much like in the affine/projective setting---use the monodromy 
action and the membership test to build up a (complete) witness collection. 
We seek a practical trace test to verify that a partial witness collection is, in fact, complete. 
That is, if we have equality  $W'_{m_1,m_2}=V'\cap\left(\MOne\times \MTwo\right)$ for each partial witness set
$W'_{m_1,m_2}$  for $V'$.  
 
By Example 20 of~\cite{HR15}, the trace of a multihomogeneous witness set as the linear subspaces $\MOne$ and 
$\MTwo$ each vary in pencils is not multilinear. 
The trace test for subvarieties of products of projective spaces in~\cite{HR15} uses the 
\demph{Segre embedding} 
$\sigma\,\colon\PP^{n_1}\times\PP^{n_2} \to \PP^{(n_1+1)(n_2+1)-1}$ to construct the proxy
witness sets as in Remark~\ref{remark:trace-image} (with $\phi = \sigma$). 
Since $\sigma$ gives an isomorphism from $V$ to $\sigma(V)$, proxy witness sets are preimages of witness sets (in contrast
to~\cite{WitnessProj} where extra work is needed, since the preimage of a
witness point may not be 0-dimensional).  
 
%%%%%%%%%%%%%%%%%%%%%%%%%%%%%%%%%%%%%%%%%%%%%%%%%%%%%%%%%%%%%%%%%%%%%%%%%%%%%%%%% 
\begin{remark}\label{remark:Segre} 
 Multihomogeneous witness sets for $V$ are typically significantly smaller than witness sets for $\sigma(V)$. 
 Let $V\subset \PP^{n_1}\times\PP^{n_2}$ be a subvariety with multidegrees $d_{m_1,m_2}$.
 By Exercise 19.2 in~\cite{Harris} the degree of its image under the Segre embedding is 
\[ 
   \deg(V)\ =\ \sum_{m_1+m_2=m} d_{m_1,m_2}\frac{m!}{m_1! m_2!}\,. 
\] 
 This is significantly larger that the union of the multihomogeneous witness sets for $V$. 
 Thus a witness set for the image of $V$ under the Segre embedding (a \demph{Segre witness set} in~\cite{HR15}) 
 involves significantly more points than any of its multihomogeneous witness sets. 
 \end{remark}  
%%%%%%%%%%%%%%%%%%%%%%%%%%%%%%%%%%%%%%%%%%%%%%%%%%%%%%%%%%%%%%%%%%%%%%%%%%%%%%%%% 
 
\begin{example} 
  The graph $V\subset\PP^m\times\PP^m$ of a general linear map has multidegrees 
  $(1,\dotsc,1)$ with sum $m{+}1$, but its image under the Segre embedding has degree $2^m$. 
  If $V$ is the closure of the graph of the the standard Cremona transformation 
  $[x_0,\dotsc,x_m]\mapsto[1/x_0,\dotsc,1/x_m]$, then its multidegrees are $d_{i,m-i}=\binom{m}{i}$ 
  with sum $2^m$   and its degree under the Segre embedding is $\binom{2m}{m}=\sum_i \binom{m}{i}^2$, which is  
  considerably larger. 
\end{example} 
 
  This suggests that one should seek algorithms that work directly with multihomogeneous 
  witness sets $W_{m_1,m_2}$ for $m_1+m_2=m$ and---as the graph of Cremona suggests---also involve as few of 
  these as possible. 
 
  Algorithm~\ref{algorithm:multi-trace-test} does exactly that while avoiding the Segre embedding. 
%%%%%%%%%%%%%%%%%%%%%%%%%%%%%%%%%%%%%%%%%%%%%%%%%%%%%%%%%%%%%%%%%%%%%%%%%%%%%%%%% 
 
%%%%%%%%%%%%%%%%%%%%%%%%%%%%%%%%%%%%%%%%%%%%%%%%%%%%%%%%%%%%%%%%%%%%%%%%%%%%%%%%% 
\section{Dimension reduction and multihomogeneous trace test}\label{S:dimension} 
 
%The following version of Bertini's theorem is particularly useful. % in numerical algebraic geometry. 
%It follows, for instance, from~\cite[Thm.~6.3~(4)]{jouanolou}. 

We give a useful version of Bertini's theorem that follows from~\cite[Thm.~6.3~(4)]{jouanolou}.
 
%%%%%%%%%%%%%%%%%%%%%%%%%%%%%%%%%%%%%%%%%%%%%%%%%%%%%%%%%%%%%%%%%%%%%%%%%%%%%%%%% 
\begin{theorem}[Bertini's Theorem]\label{thm:bertini} 
 Let $V$ be a variety and $\phi\,\colon V\ratto \PP^n$ be a rational map such that $\dim\phi(V) \geq 2$. 
 Then $V$ is irreducible if and only if $V\cap \phi^{-1}(H)$ is irreducible for a generic hypersurface
 $H\subset\PP^n$.  
\end{theorem} 
%%%%%%%%%%%%%%%%%%%%%%%%%%%%%%%%%%%%%%%%%%%%%%%%%%%%%%%%%%%%%%%%%%%%%%%%%%%%%%%%% 
 
In \S\ref{S:affine} we sliced a projective variety $V\subset\PP^n$, $\dim V \geq 2$,
with a general linear subspace.  
This reduced the dimensions of the ambient space and of the variety, but did not alter its degree or
irreducible decomposition.  
A similar dimension reduction procedure is more involved for subvarieties of a product of
projective spaces.  
 
%%%%%%%%%%%%%%%%%%%%%%%%%%%%%%%%%%%%%%%%%%%%%%%%%%%%%%%%%%%%%%%%%%%%%%%%%%%%%%%%% 
\begin{proposition}\label{prop:reduction-to-curve} 
 Let $V \subset \PP^{n_1}\times\PP^{n_2}$ be an irreducible variety 
 and suppose that $d_{m_1,m_2}(V)\neq 0$ is a nonzero multidegree  with $1\leq m_1,m_2$. 
 For $i=1,2$, let $\MiPrime$ be a general linear subspace of\/ $\PP^{n_i}$ of codimension $m_i{-}1$. 
 Then $\defcolor{V'}:=V\cap(\MOnePrime\times \MTwoPrime)$ is irreducible, has dimension two, and multidegrees 
\[ 
  d_{0,2}(V')\ =\ d_{m_1-1,m_2+1}\,,\  \ 
  d_{1,1}(V')\ =\ d_{m_1,m_2}\,,\  \quad\mbox{\rm and}\quad 
  d_{2,0}(V')\ =\ d_{m_1+1,m_2-1}\,. 
\] 
 We have several overlapping cases. 
\begin{enumerate} 
\item[(1)] If $d_{0,2}(V')=d_{2,0}(V')=0$, then $\pi_1(V')$ and $\pi_2(V')$ are both curves, $V'$ is their product, 
        and $V$ is the product of its projections $\pi_1(V)\subset\PP^{n_1}$ and  $\pi_2(V)\subset\PP^{n_2}$. 
\item[(2a)] If $d_{0,2}(V')=0$ then $\pi_1(V')$ is an irreducible curve and $V'$ is fibered over $\pi_1(V')$ by curves. 
      Also, $\pi_1(V)$ is irreducible of dimension $m_1$ and the map $V\to\pi_1(V)$ is a fiber bundle. 
      If $d_{2,0}(V')=0$, then the same holds {\it mutatis mutandis}. 
\item[(2b)] One of  $d_{2,0}(V')$ or $d_{0,2}(V')$ is non-zero. 
      Suppose that $d_{2,0}(V')\neq 0$. 
      Then $\pi_1(V')$ is two-dimensional, and for a general hyperplane $H\subset\PP^{n_1}$, 
      $W \cap (H \times \PP^{n_2})$ is an irreducible curve $C$ with $d_{1,0}(C)=d_{2,0}(V')$ and 
      $d_{0,1}(C)=d_{1,1}(V')$. 
\end{enumerate} 
\end{proposition} 
%%%%%%%%%%%%%%%%%%%%%%%%%%%%%%%%%%%%%%%%%%%%%%%%%%%%%%%%%%%%%%%%%%%%%%%%%%%%%%%%% 
 
Case (1) is distinguished from cases (2a) and (2b) as follows. 
Consider the linear maps induced by projections $\pi_i$, $i=1,2$,  
on the tangent space of $V'$ at a general point. 
We are in case (1) if and only if both maps on tangent spaces are degenerate. 
 
Case (1) reduces to the analysis of projections $\pi_i(V')$, 
otherwise it is possible to use Bertini's theorem to slice once more (preserving irreducibility and 
multidegrees)  to reduce a two-dimensional subvariety $V'$ to a curve $C$. 
 
 %%%%%%%%%%%%%%%%%%%%%%%%%%%%%%%%%%%%%%%%%%%%%%%%%%%%%%%%%%%%%%%%%%%%%%%%%%%%%%%%% 
\begin{example}
Consider the three-dimensional variety $V$ in $\PP^4\times\PP^4$ defined by 
%$f:=(x_0+x_1)^3+(x_0+x_2)^3+(x_0+x_3)^3+(x_0+x_4)^3=0$ 
$f:=\sum_{i=1}^4(x_0+x_i)^3=0$ 
and the maximal minors of the $5\times 2$ matrix 
$[y_i,\partial f/\partial x_i]$.
The multidegree of $V$ is $(d_{3,0},d_{2,1},d_{1,2},d_{0,3})=(3,6,12,0)$.
Since $d_{2,1}(V)\neq 0$, we intersect $V$ 
with $\MOnePrime\times \MTwoPrime$ where $\MOnePrime$ is a hyperplane in $\PP^4$ and $\MTwoPrime=\PP^4$;
the multidegree of $V'$ is $(d_{2,0},d_{1,1},d_{0,2})=(3,6,12)$.
On the other hand, $d_{1,2}(V)$ is also non-zero. 
Intersecting $V$ with $\MOnePrime\times \MTwoPrime$ where now  $\MOnePrime=\PP^4$ and $\MTwoPrime$ is a
hyperplane in $\PP^4$, the multidegree of  $V'$ is $(d_{2,0},d_{1,1},d_{0,2})=(6,12,0)$.
Each may be sliced once more to reduce to a curve in either $\PP^4\times\PP^2$, $\PP^3\times\PP^3$, or
$\PP^2\times\PP^4$. 
\end{example}
%%%%%%%%%%%%%%%%%%%%%%%%%%%%%%%%%%%%%%%%%%%%%%%%%%%%%%%%%%%%%%%%%%%%%%%%%%%%%%%%% 

The following %proposition 
multihomogeneous counterpart  of Proposition~\ref{prop:generic-projection-to-P2}
is not a part of our multihomogeneous trace test. 
We include it to provide better intuition to the reader. 
%It is the multihomogeneous counterpart of the reduction to a planar curve stated in
%Proposition~\ref{prop:generic-projection-to-P2}.  

%%%%%%%%%%%%%%%%%%%%%%%%%%%%%%%%%%%%%%%%%%%%%%%%%%%%%%%%%%%%%%%%%%%%%%%%%%%%%%%%% 
\begin{proposition}\label{prop:generic-projection-to-P1xP1} 
Let $C\subset \PP^{n_1}\times\PP^{n_2}$ be a curve. 
Let $\alpha_i\,\colon\PP^{n_i}\ratto \PP^1$ be a generic linear projection for $i=1,2$. 
Then $C$ is irreducible if and only if $(\alpha_1\times\alpha_2)(C) \subset \PP^1\times\PP^1$ is irreducible. 
\end{proposition} 
%%%%%%%%%%%%%%%%%%%%%%%%%%%%%%%%%%%%%%%%%%%%%%%%%%%%%%%%%%%%%%%%%%%%%%%%%%%%%%%%% 

%Now that we have 
Having 
reduced to %the case of 
a curve $C \subset \PP^1\times\PP^1$, we could 
use a trace test via the Segre embedding $\PP^1\times\PP^1\to \PP^3$ as %discussed 
in Remark~\ref{remark:Segre}. 
It is more direct to use the trace test in $\CC^2$.
 
%%%%%%%%%%%%%%%%%%%%%%%%%%%%%%%%%%%%%%%%%%%%%%%%%%%%%%%%%%%%%%%%%%%%%%%%%%%%%%%%% 
\begin{example}\label{Ex:P1xP1} 
%%%%%%%%%%%%%%%%%%%%%%%%%%%%%%%%%%%%%%%%%%%%%%%%%%%%%%%%%%%%%%%%%%%%%%%%%%%%%%%%%

%%%%%%%%%%%%%%%%%%%%%%%%%%%%%%%%%%%%%%%%%%%%%%%%%%%%%%%%%%%%%%%%%%%%%%%%%%%%%%%%%
Let us consider the trace test for a curve $C$ in $\PP^1\times\PP^1$. 
%To avoid the use of multiple indices with $z^{(i)}$, we
%
%  Frank thinks that we can just drop the z^ stuff, at least here. Jose agrees. 
% 
Let $x:=(x_0,x_1)$ and $y:=(y_0,y_1)$ be homogeneous coordinates on the two copies of $\PP^1$.  
%J% Do we want to be consistent and have z1=x, z2=y? Jose likes using the x,y. 
Let $C$ be a curve given by the bihomogeneous polynomial $f(x,y):=x_0 y_0^2 - x_1 y_1^2$ of bidegree $(1,2)$

Linear forms $\ellx:=x_1-\frac{7}{2}x_0$ and $\elly:=y_1+y_0$ cut out witness sets  $W_{1,0}$ and $W_{0,1}$ for $C$.
Choose the (sufficiently general) linear forms
%
% Frank removed the unnecessary equation numbering
%
 \[
     \hx\ :=\ x_0\,,\quad
     \hy\ :=\ y_0\,,\quad
    \ellxPrime\ :=\ \tfrac{3}{2}x_1-x_0\,,
    \quad\text{ and }\quad
    \ellyPrime\ :=\ -\tfrac{4}{3}y_1-\tfrac{10}{3}y_0\,,
 \]
 and consider the bilinear form
 \[
   g(x,y)\ :=\ \hx\ellyPrime+\ellxPrime \hy +\hx\hy\,.
 \]
Following the points of $W_{1,0}\cup W_{0,1}$ along the homotopy
 \begin{equation}\label{Eq:bilinear_htpy}
    \defcolor{h(t)}\ :=\ (1-t)\ellx\elly\ +\ tg  
 \end{equation}
from $t=0$ to $t=1$ gives the three $(=1+2)$ points of $C\cap\V{g}$.

Then $(x_1,y_1)$ provides coordinates in the affine chart where $\hx=1$ and $\hy=1$,
with multihomogeneous witness sets $W_{1,0}=\{(1,1)\}$ and $W_{0,1}=\{(7/2,-\sqrt{2/7}),(7/2,\sqrt{2/7})\}$.
The homotopy~\eqref{Eq:bilinear_htpy} from $t=0$ to $t=1$ takes the three witness points 
$W_{1,0}\cup W_{0,1}$ for $C\cap\V{\ellx\elly}$ to the three witness points for 
$C\cap\V{g}$. 
In this chart, $\V{g}$ is a line in $\CC^1\times\CC^1 = \CC^2$, so that $C\cap\V{g}$ is a witness set for the
curve $C$ in $\CC^2$.  
%%%%%%%%%%%%%%%%%%%%%%%%%%%%%%%%%%%%%%%%%%%%%%%%%%%%%%%%%%%%%%%%%%%%%%%%%%%%%%%%%
\begin{figure}[htb]
\begin{picture}(183,150)(-23,0)
 \put(0,0){\includegraphics[height=150pt]{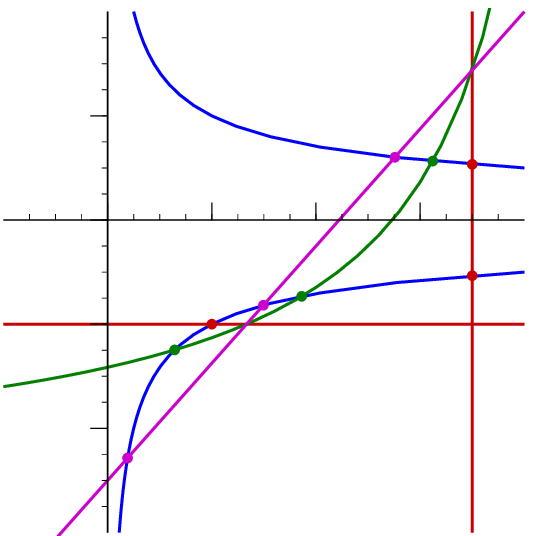}}
 \put(10,26.5){$-2$}   \put(18,116){$1$}  \put(18,145){$y_1$}
 \put(57,97){$1$} \put(87,97){$2$}  \put(146,94){$x_1$}
 \put(-13,60){$\elly$} \put(-23,37){$h(\tfrac{1}{2})$} \put(55,35){$g$}
 \put(137,27){$\ellx$}  \put(150,68){$C$} \put(147,107){$C$}
\end{picture}
\qquad\qquad
\begin{picture}(180,150)
 \put(0,0){\includegraphics[height=150pt]{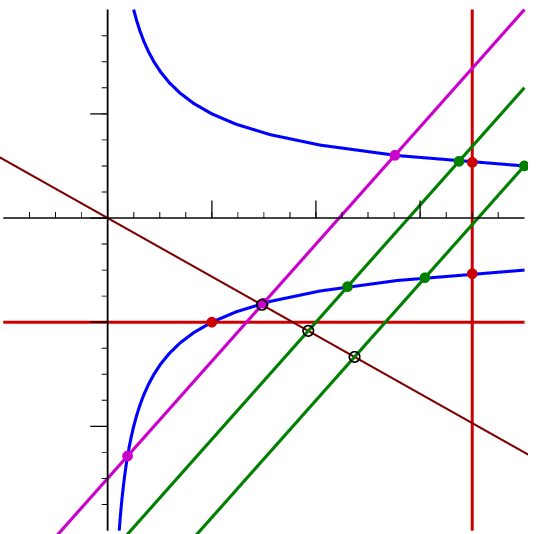}}
 \put(10,26.5){$-2$}   \put(18,116){$1$}  \put(18,145){$y_1$}
 \put(57,97){$1$} \put(87,97){$2$}  \put(-6,95){$x_1$}
 \put(115,122){$g$}
 \put(151,146){$\tau=0$}
% \put(150,136){$\tau=-0.5$}
 \put(151,124){$\tau=-1$}
% \put(150,114){$\tau=-1.5$}
 \put(151,102){$\tau=-2$}
% \put(150,92){$\tau=-2.5$}
\end{picture}
%\caption{Curves, witness sets, and the trace test}
\caption{On the left: the red lines, green curve, and magenta curve correspond, respectively, to
  \eqref{Eq:bilinear_htpy} at $t=0,\,\frac{1}{2},\,1$. 
On the right: the parallel slices $\V{g+\tau}$ are in green, and the average of the witness points ($\frac{1}{3}$ of the
trace) lies on the brown line.
The blue curve is $C$.}
\label{F:bilinear}
\end{figure}
%%%%%%%%%%%%%%%%%%%%%%%%%%%%%%%%%%%%%%%%%%%%%%%%%%%%%%%%%%%%%%%%%%%%%%%%%%%%%%%%%
%[[1.481481482, -.8333333333], 
%[1.703703704, -.9583333333], 
%[1.925925926, -1.083333334], 
%[2.148148148, -1.208333334], 
%[2.370370370, -1.333333333], 
%[2.592592593, -1.458333333]]

Using the witness points $C\cap\V{g}$, we perform the trace test for $C$ in this affine chart,
using the family of lines, $\V{g+\tau}$ as $\tau$ varies. The values of the trace at three points,
\[
  \begin{array}{r|c|c|c}
%    \tau & 0 & -0.5 &-1 &-1.5 & -2 & -2.5\\\hline
%    \text{avg }x_{1} &1.48148&1.70370&1.92592&2.148148&2.37037&2.59259\\
%    \text{avg }y_{1} &-.83333&-.958333&-1.08333&-1.208333&-1.33333&-1.458333
    \tau & 0 & -1 & -2 \\\hline
    \text{avg }x_{1} &1.48148&1.92592&2.37037\\
    \text{avg }y_{1} &-.83333&-1.08333&-1.33333
  \end{array}
\]
let us find the trace, $(\frac{40}{27}-\frac{4}{9}\tau, -\frac{5}{6}+\frac{1}{4}\tau)$, by interpolation.
\end{example} 
 
\begin{remark}\label{remark:trace-test-for-curve} 
It is not essential to reduce to a curve in $\PP^1\times\PP^1$. 
The construction and argument of Example~\ref{Ex:P1xP1} holds, {\it mutatis mutandis}, for an irreducible
curve $\PP^{n_1}\times\PP^{n_2}$  with the trace test performed in an affine patch 
$\CC^{n_1+n_2} \simeq \CC^{n_1}\times\CC^{n_2}$. 
\end{remark} 
%%%%%%%%%%%%%%%%%%%%%%%%%%%%%%%%%%%%%%%%%%%%%%%%%%%%%%%%%%%%%%%%%%%%%%%%%%%%%%%%% 

%To summarize, w
We give a high-level description of an algorithm for the trace test for a \demph{collection of partial 
multihomogeneous witness sets}. 
Details and improvements of a numerical irreducible decomposition algorithm that uses this trace test shall be
given elsewhere. 
For an overview of numerical irreducible decomposition see Section 6 of \cite{HR15}. 
 
Let us fix\vspace{-3pt} 
\begin{description} 
\item[dimension] an integer $m$, the dimension of a witnessed component; 
\item[affine charts] for $i=1,2$, linear forms $h^{(i)}$ defining affine charts $h^{(i)}=1$ in $\PP^{n_i}$; 
\item[slices] for $i=1,2$, for $j=1,\ldots,m$, linear forms $\ell^{(i)}_j$ defining hyperplanes in $\PP^{n_i}$; 
\end{description} 
Write $L_{m_1,m_2}$ for the system 
  $\{h^{(1)}-1,\ell^{(1)}_1,\ldots,\ell^{(1)}_{m_1},h^{(2)}-1,\ell^{(2)}_1,\ldots,\ell^{(2)}_{m_2}\}$. 
 Observe that the system $L_{m_1,m_2}$ defines a product $\MOne\times \MTwo$ in an affine 
chart of $\PP^{n_1}\times\PP^{n_2}$. 

\begin{algorithm}[Multihomogeneous Trace Test]  \label{algorithm:multi-trace-test} 
\label{algorithm:monodromy-breakup} 
\ \newline 
\noindent{\sc Input:} \vspace{-3pt} \begin{description} 
  \item[equations] a multihomogeneous polynomial system $F$; 
  \item[a partial witness collection]  partial witness sets $W_{m_1,m_2}$ where $m_1=0,\ldots,m$ and $m_2=m-m_1$ 
representing an irreducible component $V\subset\V{F}$, i.e., $W_{m_1,m_2} \subset V\cap\V{L_{m_1,m_2}}$. 
\end{description} 
\noindent{\sc Output:} a boolean value = $\text{the witness collection is complete}$. 
\begin{algorithmic}[1] 
\IF{$W_{m_1,m_2}=\emptyset$ for all $m_1=0,\ldots,m$ but one} 
  \IF{both projections of $\V{F}$ to the factors $\PP^{n_i}$ are degenerate at an available witness point} 
    \RETURN (both trace tests for the projections to $\PP^{n_i}$ for $i=1,2$ pass) \AND (the unique nonempty
    set of witness points equals the product of its projections) % to the factors). 
  \ELSE \RETURN \FALSE 
  \ENDIF 
\ELSE 
  \FOR{$m_1 = 0,\ldots, m-1$} 
  \IF{the trace test in $\CC^{n_1+n_2}$ described in Example~\ref{Ex:P1xP1} and 
    Remark~\ref{remark:trace-test-for-curve} (after tracking $W_{m_1,m_2}$ and $W_{m_1+1,m_2-1}$ along the 
    deformation from $\ell^{(1)}_{m_1+1}\ell^{(2)}_{m_2}$ to a general affine linear function on 
    $\CC^{n_1+n_2}$) does not pass} \RETURN \FALSE \ENDIF 
  \ENDFOR 
  \RETURN \TRUE 
\ENDIF 
\end{algorithmic} 
\end{algorithm} 
 
We presented results for subvarieties of a product of {\em two} projective spaces for the sake of clarity. 
These arguments generalize to a product of arbitrarily many~factors with a few subtleties. 
 
%%%%%%%%%%%%%%%%%%%%%%%%%%%%%%%%%%%%%%%%%%%%%%%%%%%%%%%%%%%%%%%%%%%%%%%%%%%%%%%%% 
 
\section{Proofs}\label{S:proofs} 
We present a proof of Proposition~\ref{prop:generic-projection-to-P1xP1} 
immediately following  the proof of Proposition~\ref{prop:generic-projection-to-P2}. 
While the first is standard, it helps to better understand the second.
A map on a possibly reducible variety is birational if it is an isomorphism on a dense open set.
 
Every surjective linear map $\PP^n\ratto \PP^{n-1}$, $n>1$, is the projection from a point 
$p\in\PP^n\simeq \Proj(\CC^{n+1})$. 
Namely, it is the projectivization $\alpha_p\,\colon\PP^n\ratto \PP^{n-1}$ of the quotient map 
$\CC^{n+1}\to\CC^{n+1}/\CC p \simeq\CC^n$. 
This rational map is not defined at $\CC p$. 

%%%%%%%%%%%%%%%%%%%%%%%%%%%%%%%%%%%%%%%%%%%%%%%%%%%%%%%%%%%%%%%%%%%%%%%%%%%%%%%%%
\begin{proof}[Proof of Proposition~\ref{prop:generic-projection-to-P2}] 
We argue that a projection from a generic point is a birational map from a curve $C\subset \PP^n$ to 
its image in $\PP^{n-1}$ for $n\geq 3$.
Birational maps preserve (ir)reducibility.

Consider the incidence variety of triples 
$(p,c,c')\subset \PP^n \times C\times C$, where 
$p, c, c'$ are collinear.   
Projecting to  $C\times C$ shows that this incidence variety is three-dimensional because the image is two-dimensional and
generic fiber is one-dimensional.  
Moreover, the projection to $\PP^n$ is dense in the secant variety of $C$. 

When $n=3$, observe that this secant variety is either
(1) not dense, so projecting from a point not in 
     its closure is a birational map from $C$ onto a plane 
     curve, or 
(2) dense.  In case (2), a general point $p \in \PP^3$ has finitely many preimages $(p,c,c')\in \PP^n \times
C\times C$, so the projection $\alpha_p\,\colon\PP^3\ratto \PP^{2}$  
     gives a birational map from $C$ to a plane curve $C'$ with finitely many 
     points of self-intersection.  
 
Note that for $n>3$ only case (1) is possible. 

Thus, we are always able to reduce the ambient dimension by one until $n=2$. 
\end{proof} 
 
\begin{proof}[Proof of Proposition~\ref{prop:generic-projection-to-P1xP1}] 
Assume that $n=n_1\geq n_2$. 
Any inclusion $\PP^{n_2}\hookrightarrow\PP^{n_1}=\PP^n$ gives an isomorphism from $C$ to a curve in $\PP^n\times\PP^n$.
We may replace $\alpha_2$ with a generic linear map $\PP^n\ratto\PP^1$ that it factors through. 
 
For $(p,q) \in \PP^n\times\PP^n$ consider the product of projection-from-a-point maps 
\[
   \alpha_p\times \alpha_q\ \colon\ \PP^n\times\PP^n \ratto \PP^{n-1}\times\PP^{n-1}\,.
\]
 Let $\Gamma$ be the incidence variety of triples 
$(s,c,c')\in (\PP^n\times\PP^n) \times C\times C$, 
where $s=(s_1,s_2)$, $c=(c_1,c_2)$, and $c'=(c'_1,c'_2)$ 
such that $s_i,c_i,c_i'$, are collinear for $i=1,2$. 

The projection of $\Gamma$ to $C\times C$ has fibers $\PP^1\times \PP^1$, so it is four-dimensional. 
The projection to $\PP^n\times\PP^n$ is dense in a generalized secant variety of dimension four. 
 
When $n=2$, either this secant variety is 
(1) dense, or it is  
(2) not dense, so that ${\alpha_p\times \alpha_q}$ for a point $(p,q)$ not in 
     its closure is a birational map from $C$ to its image.
 In case (1), a general point  $(p,q) \in\PP^n\times\PP^n$ 
     has finitely many preimages. This implies that the map 
${\alpha_p\times \alpha_q}$ is one-to-one on $C$ with the exception of finitely many points, whose images are
self-intersections of the curve $(\alpha_p\times \alpha_q)(C)$.  
 
For $n>2$, case (2) is the only possibility.
 
Thus we are always able to reduce $n$ by one until $n=1$. 
\end{proof}

%%%%%%%%%%%%%%%%%%%%%%%%%%%%%%%%%%%%%%%%%%%%%%%%%%%%%%%%%%%%%%%%%%%%%%%%%%%%%%%%% 
\begin{proof}[Proof of Proposition~\ref{Prop:two}] 
 Let $W$ be a subset of the fiber $C_t$ of $C$ over 
 $t\in\ell\smallsetminus\Delta$ whose sum $s(t)$ is an affine linear function of $t$. 
Note that local linearity in a neighborhood of some $t$ implies global linearity: in particular, an analytic
continuation  along any loop $\gamma\,\colon[0,1]\to\ell\smallsetminus\Delta$ with $\gamma(0)=\gamma(1)=t$
does not change  the value of $s(t)$. 
 
 Following points of $C_t$ along the loop above $\gamma$ gives a permutation of $C_t$. 
 By our assumption of general position 
 and~\cite[Lemma on page 111]{Arbarello-Cornalba-Griffiths-Harris}, 
 every permutation of $C_t$ is obtained by some loop $\gamma$. 
 
 Suppose that $W$ is a proper subset of $C_t$. 
 Then there is a point $u\in W$ and a point $v\in C_t\smallsetminus W$, hence $u\neq v$. 
 Let $\gamma$ be a loop in $\ell\smallsetminus\Delta$ based at $t$ whose permutation interchanges $u$ and 
 $v$ and fixes the other points of $C_t$. 
 In particular, $u(\gamma(1))=v$. 
 Since $s(t)=\sum_{w\in W} w(t)$ has the same value at the beginning and the end of the loop, we have  
\[ 
   \sum_{w\in W} w(\gamma(0))\ =\ \sum_{w\in W} w(\gamma(1))\,. 
\]  
Taking the difference gives $0=u(\gamma(1))-u(\gamma(0))$ so that $u=v$, a contradiction. 
\end{proof} 
%%%%%%%%%%%%%%%%%%%%%%%%%%%%%%%%%%%%%%%%%%%%%%%%%%%%%%%%%%%%%%%%%%%%%%%%%%%%%%%%% 
 
\begin{proof}[Proof of Proposition~\ref{prop:reduction-to-curve}] 
Note that projections $\pi_i\,\colon\PP^{n_1}\times \PP^{n_2}\to\PP^{n_i}$, for $i=1,2$, satisfy the assumptions
on the map $\phi$ in Theorem~\ref{thm:bertini}.  
Applying the theorem $m_i-1$ times for $\pi_i$, for $i=1,2$, gives the proof of the first part of the conclusion. 
 
The rest of the conclusion follows from the case analysis: in the case $\dim\pi_1(V')\geq 2$, one more application of
Theorem~\ref{thm:bertini} for the map $\pi_1$ proves the statement.  
\end{proof} 
 
%%%%%%%%%%%%%%%%% bibliography %%%%%%%%%%%%%%%%%%%%%%%%%%%%%%%%%%%%%%%% 
\providecommand{\bysame}{\leavevmode\hbox to3em{\hrulefill}\thinspace}
\providecommand{\MR}{\relax\ifhmode\unskip\space\fi MR }
\providecommand{\MRhref}[2]{%
  \href{http://www.ams.org/mathscinet-getitem?mr=#1}{#2}
}
\providecommand{\href}[2]{#2}

\end{document}